\documentclass{article}

\usepackage{amsthm,amsmath,amssymb,amstext,amscd,amsfonts,amsbsy,amsxtra,latexsym}
\usepackage{fullpage}
\usepackage[top=3cm,bottom=3cm,right=3cm,left=3cm]{geometry}

\newcommand{\field}[1]{\mathbb{#1}}

\newcommand{\N}{\field{N}}
\newcommand{\Z}{\field{Z}}
\newcommand{\R}{\field{R}}
\newcommand{\C}{\field{C}}

\numberwithin{equation}{section}
\newtheorem{theorem}{Theorem}[section]
\newtheorem{lemma}[theorem]{Lemma}
\newtheorem{corollary}[theorem]{Corollary}

\newtheorem{proposition}[theorem]{Proposition}
\theoremstyle{remark}

\newtheorem*{definition}{Definition}

\renewenvironment{proof}[1][Proof]{\begin{trivlist}
\item[\hskip \labelsep {\bfseries #1:}]}{\qed\end{trivlist}}

\title{On a generalisation of Roth's theorem for arithmetic progressions and applications to sum-free subsets}
\author{Jehanne Dousse}
\date{\today}

\begin{document}

\maketitle

%
%

\begin{abstract}
We prove a generalisation of Roth's theorem for arithmetic progressions to $d$-configurations, which are sets of the form $\{n_i+n_j+a\}_{1\leq i \leq j \leq d}$ with $a,n_1,...,n_d \in \N$, using Roth's original density increment strategy and Gowers uniformity norms. Then we use this generalisation to improve a result of Sudakov, Szemer\'edi and Vu about sum-free subsets~\cite{Sudakov} and prove that any set of $n$ integers contains a sum-free subset of size at least $\log n \left(\log ^{(3)} n \right)^{1/32772 - o(1)}$.
\end{abstract}

\section{Introduction}
In 1953 Roth~\cite{Roth} proved his famous theorem about arithmetic progressions of length~$3$.
\begin{theorem}
\label{roth}
Let $0 < \alpha < 1$. Any subset of $\{1, \dots, N \}$ of size $\alpha N$ with $N \geq N_0(\alpha)$ contains a non-trivial arithmetic progression of length $3$.
\end{theorem}
A \emph{$d$-configuration} is a set of the form $\{n_i + n_j + a\}_{1 \leq i \leq j \leq d}$ with $a, n_1, ..., n_d \in \N$. The $d$-configuration is \emph{non-trivial} if for all $i \neq j$, $n_i \neq n_j$.

Our aim is to prove the following generalisation of Roth's theorem:
\begin{theorem}
\label{th1}
Let $0 < \alpha < 1$ and $d \geq 1$. Any subset of $\{1, \dots, N \}$ of size $\alpha N$ with $N \geq N_0(\alpha,d)$ contains a non-trivial $d$-configuration.
\end{theorem}
If $d=2$, this is equivalent to Roth's theorem because $2$-configurations are exactly arithmetic progressions of length $3$.

Theorem~\ref{th1} can be proved easily by using Szemer\'edi's theorem~\cite{Sze75}, presented here in a quantitative version due to Gowers~\cite{Gowers}.
We write $p \uparrow q$ for $p^q$, with the convention that $p \uparrow q \uparrow r$ stands for $p \uparrow (q \uparrow r)$.

\begin{theorem}
\label{szeme}
Let $0 < \alpha < 1$, and $k$ be a positive integer. Let $N \geq 2 \uparrow 2 \uparrow (\alpha)^{-1} \uparrow 2 \uparrow 2 \uparrow (k+9)$ and let $A$ be a subset of $[N]$ with cardinality $\alpha N$. Then $A$ contains a non-trivial arithmetic progression of length $k$.
\end{theorem}

Now to prove Theorem~\ref{th1} we can locate in $A$ a progression $P$ of length $2d-1$, set the $2 n_i + a$'s to be the elements with odd indices of $P$, and notice that $P$ is a $d$-configuration. By Theorem~\ref{szeme}, this is possible if $N \geq 2 \uparrow 2 \uparrow (\alpha)^{-1} \uparrow 2 \uparrow 2 \uparrow (2d+8)$, but this bound is very unsatisfactory. Therefore we give a proof that does not involve such a deep theorem as Szemer\'edi's, and show that Theorem~\ref{th1} is true for $N \geq \exp \left(\exp \left( \left(\frac{C}{\alpha}\right)^{d(d+1)-1}\right)\right),$ where $C$ is some absolute constant.

\ 

In the second part of this paper we use Theorem~\ref{th1} to prove a result about sum-free subsets.
For two finite sets of real numbers $A$ and $B$, one says that $B$ is \emph{sum-free} with respect to $A$ if the set $\{b + b' | b,b' \in B, b \neq b' \}$ is disjoint from $A$. Let $\phi (n)$ denote the largest integer such that any set $A$ of size $n$ contains a subset of cardinality $\phi (n)$ which is sum-free with respect to A. An interesting question is to find a lower bound for $\phi(n)$. Erd\H{o}s first mentioned in~\cite{Erdos} that $\phi(n) \geq c \log n$ for some constant $c$, and the first published proof of this result was done by Choi~\cite{Choi} who proved that $\phi (n) \geq \log_2 n$. Then Ruzsa~\cite{Ruzsa} improved this result slightly by showing that $\phi (n) \geq 2\log_3 n -1$. Recently Sudakov, Szemer\'edi and Vu~\cite{Sudakov} gave the first superlogarithmic bound by showing that $\phi(n) \geq \log n \left(\log ^{(5)} n \right)^{1- o(1)}$, where $\log^{(i)}n$ denotes the iterate logarithm ($\log^{(1)}x =x$, $\log^{(i+1)}x = \log(\log^{(i)}x)$). By modifying a small part of their proof using Theorem~\ref{th1}, we prove that $\phi(n) \geq \log n \left(\log ^{(3)} n \right)^{1/32772 - o(1)}$.

\section{Proof of Theorem~\ref{th1}}
As in several proofs of Roth's theorem~\cite{Roth,Green,Heath,Sze2}, we use the density increment strategy to prove Theorem~\ref{th1}.
It consists in showing that either $A$ contains a non-trivial $d$-configuration or it has increased density on some arithmetic progression.

In the following, we will use the notations $[N]:= \{1,...,N\}$ and $e(\theta):= e^{2i\pi \theta}.$


\subsection{Complexity and Gowers uniformity norms}
Gowers' uniformity norms play an important role in his proof of Szemeredi's theorem~\cite{Gowers}, and therefore in the particular case of Roth's theorem, as explained in Green's course notes~\cite{Green}.

We will use these norms in our proof too, but we first need the notion of complexity, introduced by Green and Tao in \cite{GreenTao}.

\begin{definition}
Let $\Psi =(\psi_1,...,\psi_t)$ be a system of affine-linear forms. If $1 \leq i \leq t$ and $s \geq 0$, we say that $\Psi$ has $i$-complexity at most $s$ if one can cover the $t-1$ forms $\{\psi_j : j \in [t]\setminus \{i\} \}$ by $s+1$ classes, such that $\psi_i$ does not lie in the affine-linear span of any of these classes. The \emph{complexity} of $\Psi$ is defined to be the least $s$ for which the system has $i$-complexity at most $s$ for all $1 \leq i \leq t$, or $\infty$ if no such $s$ exists.
\end{definition}

We will now define the Gowers uniformity norms.

Let $f$ be a function from $\Z/N\Z$ to $\C$. The \emph{expectation} of $f$ is defined to be the quantity $$ \mathbb{E}_{x \in \Z/N\Z} f(x) := \frac{1}{N} \sum_{x \in \Z/N\Z} f(x).$$

\begin{definition} Suppose that $f : \Z/N\Z \rightarrow \C$ is a function. Let $k \geq 2$ be an integer. The \emph{Gowers $U^k$-norm} is defined by $$ \| f \|_{U^k} := \left(\mathbb{E}_{x,h_1,...,h_k \in \Z/N\Z} \prod_{(\omega_{1},...,\omega_k) \in \{0,1\}^k} C^{| \omega_1 | + ... + |\omega_k|} f(x + \omega_1 h_1 +...+ \omega_k h_k) \right)^{1/2^k},$$
where we use the notation $Cf := \overline{f}.$
\end{definition}

In particular the Gowers $U^2$-norm which we will use later is defined by $$ \| f \|_{U^2} := \left(\mathbb{E}_{x,h_1,h_2 \in \Z/N\Z} f(x) \overline{f(x+h_1)} \overline{f(x+h_2)} f(x+h_1+h_2) \right)^{1/4}.$$

To establish a link between complexity and Gowers uniformity norms, we will use the following theorem \cite{GreenTao}, which allows to control systems of complexity $s$ by the Gowers $U^{s+1}$-norm.

\begin{theorem}
\label{neumann}
Let $f_1,...,f_t : \Z/N\Z \rightarrow \R$ be functions such that $|f_i(x)| \leq 1$ for all $i \in [t]$ and all $x \in \Z/N\Z$. Suppose that $\Psi =(\psi_1,...,\psi_t)$ is a system of affine-linear forms of complexity $s$ consisting of $t$ forms in $d$ variables. Then
\begin{equation*}
\mid \mathbb{E}_{x_1,...,x_d \in \Z/N\Z} \prod_{i=1}^{t} f_i(\psi_i(x_1,...,x_d)) \mid \leq \min_{1 \leq i \leq t} \| f_i \| _{U^{s+1}}.
\end{equation*}
\end{theorem}

We will now show that $d$-configurations have complexity $1$, in order to be able to control averages using the Gowers $U^2$-norm.

\begin{lemma}
Let $(n_i + n_j + a)_{1 \leq i \leq j \leq d}$ be a $d$-configuration. Then it has complexity $1$.
\end{lemma}
\begin{proof}
Let $i_0, j_0 \in [d]$ such that $i_0 < j_0$. We take the first class to consist of all forms involving $n_{i_0}$, and the second one to consist of all other forms. Then the form $n_{i_0} + n_{j_0}+ a$ is in the linear span of neither of those classes, because the first one does not involve $n_{j_0}$ at all, and the second one does not involve $n_{i_0}$.

Let us now consider the form $2 n_{i_0} +a$. As in the previous case, we take the first class to consist of all forms involving $n_{i_0}$, and the second one to consist of all other forms. Clearly $2 n_{i_0} +a$ is not in the linear span of the second class because it does not involve $n_{i_0}$ at all. And it is also not in the linear span of the first class, because in each of its forms, $n_{i_0}$ appears with a different $n_j$, and we cannot cancel them to obtain $2 n_{i_0} +a$.

Therefore the system has $(i_0,j_0)$-complexity $1$ for all $1 \leq i_0 \leq j_0 \leq d$, so the $d$-configuration has complexity at most $1$.
It does not have complexity $0$ because if we consider the form $n_{i_0} + n_{j_0}+ a$ and put all other forms in the same class, it is in its linear span. Indeed for example $2(n_{i_0}+n_{j_0}+a)=(2n_{i_0}+a)+(2n_{j_0}+a)$.
So the $d$-configuration has complexity $1$.
\end{proof}

As a particular case of Theorem~\ref{neumann}, using the fact that $d$-configurations have complexity $1$, we have the following theorem:

\begin{theorem}
\label{config}
Let $f_{ij}: \Z/N\Z \rightarrow \R$ be functions such that $|f_{ij}(x)| \leq 1$ for all $1 \leq i \leq j \leq d$, and $a$ be an integer. Then
\begin{equation*}
\mid \mathbb{E}_{n_1,...,n_d \in \Z/N\Z} \prod_{1 \leq i \leq j \leq d} f_{ij}(n_i + n_j +a) \mid \leq \min_{1 \leq i \leq j \leq d} \| f_{ij} \| _{U^{2}}.
\end{equation*}
\end{theorem}

\subsection{Obtaining a large Gowers $U^2$-norm}
In this subsection, we prove that either $A$ contains non-trivial $d$-configurations or a particular function $f_A$ has a large Gowers $U^2$-norm.

Let $P$ be an arithmetic progression of length $N$, and $A \subseteq P$ a set of size $\alpha N$.
By linear rescaling, we may assume that $P= [N]$, because it does not change either the density $\alpha$ or the number of $d$-configurations in $A$. Indeed let us assume that the forms $(n_i + n_j + a)_{1 \leq i \leq j \leq d}$ are located in the progression $P$. Adding some constant to each element of $P$, and changing $a$ accordingly, we may assume that $P$ is equal to $\{k,2k,...,Nk\}$ for some integer $k$ and that for every $1 \leq i \leq j \leq d$, $n_i + n_j + a = y_{i,j}k$. By linearly rescaling $P$ to $[N]$, our $d$-configuration becomes $(y_{i,j})_{1 \leq i \leq j \leq d}$. Let us show that this is still a $d$-configuration. We have for every $1 \leq i \leq d-1$, $n_{i+1} - n_i = (y_{i,i+1} - y_{i,i})k$, which means that the $n_i$'s are also inside an arithmetic progression with common difference $k$. Let us write $n_i = \beta_i k +b$ where $\beta_i \in \N$ for all $1 \leq i \leq d$. Then we have $n_i + n_j + a = \beta_i k + \beta_j k + 2b + a = y_{i,j}k$, so if we set $a':= \frac{2b+a}{k}$, we have $\beta_i + \beta_j + a'=y_{i,j}$ for all $1 \leq i \leq j \leq d$, which means that we still have a $d$-configuration.

Conversely if $(m_i+m_j+a)_{1 \leq i \leq j \leq d}$ is a $d$-configuration in $[N]$, then the linear system $(k(m_i + m_j +a) + b)_{1 \leq i \leq j \leq d} =(k m_i + k m_j + (ak + b))_{1 \leq i \leq j \leq d}$ is a $d$-configuration located in the progression $\{k+b,...,Nk+b\}$.

\ \\
Let us set $N':= 2N +1$ and let $\tilde{A}$ denote $A$ considered as a subset of $\Z / N'\Z$. The number of $d$-configurations in $\tilde{A}$ is the same as that in $A$, so we will identify $\tilde{A}$ and $A$.
Let us also note that to count once and only once each $d$-configuration $(n_i + n_j + a)_{1 \leq i \leq j \leq d}$ we shall assume that $a$ is equal either to $0$ or $1$.

Given functions $(f_{ij})_{1 \leq i \leq j \leq d}$, set
\begin{equation*}
\Pi_d ((f_{ij})_{1 \leq i \leq j \leq d}) := \mathbb{E}_{n_1,...,n_d \in \Z / N'\Z} \prod_{1 \leq i \leq j \leq d} f_{ij}(n_i +n_j) + \mathbb{E}_{n_1,...,n_d \in \Z / N'\Z} \prod_{1 \leq i \leq j \leq d} f_{ij}(n_i +n_j+1).
\end{equation*}

The quantity $\Pi_d(1_A,...,1_A)$, where $1_A$ is the characteristic function of $A$ ($1_A(x)=1$ if $x \in A$, $0$ otherwise), is equal to $1/N'^d$ times the number of $d$-configurations in $A$, including the trivial ones.
We shall compare this with $\Pi_d(\alpha 1_{[N]},...,\alpha 1_{[N]})$, where $\alpha 1_{[N]}(x)$ is defined to be $\alpha$ if $x \in [N]$ and $0$ if $x \in \Z / N'\Z \setminus [N]$.

To compute the difference between the two we introduce the \emph{balanced function} of $A$, defined by $$f_A:= 1_A - \alpha 1_{[N]}.$$

Let us note that the expectation of $f_A$ is equal to $0$. This property will be useful later in the proof.

Since $\Pi_d$ is multilinear, we may expand $\Pi_d(1_A,...,1_A)$ as a main term $\Pi_d(\alpha 1_{[N]},...,\alpha 1_{[N]}) = \alpha^{d(d+1)/2} \Pi_d(1_{[N]},...,1_{[N]})$ plus $2^{d(d+1)/2}-1$ other terms $\Pi_d ((g_{ij})_{1 \leq i \leq j \leq d})$ where at least one the $g_{ij}$'s is equal to $f_A$.

\begin{lemma}
\label{lemma1}
Suppose that $N > 16^d \alpha^{-\frac{d(d+1)}{2}}$ and that $A$ contains fewer than $\frac{\alpha^{d(d+1)/2} N^d}{2^{d-1} + 1}$ non-trivial $d$-configurations. Then there are $1$-bounded functions $(g_{ij})_{1 \leq i \leq j \leq d}$, at least one of which being equal to $f_A$, such that $$\mid \Pi_d((g_{ij})_{1 \leq i \leq j \leq d}) \mid \geq \left(\frac{\alpha}{C}\right)^{\frac{d(d+1)}{2}}$$ for some absolute constant $C>0$, where $\Pi_d((g_{ij})_{1 \leq i \leq j \leq d})$ is one of the $2^{d(d+1)/2}-1$ other terms.
\end{lemma}

\begin{proof}
We have
\begin{equation*}
\alpha^{\frac{d(d+1)}{2}} \Pi_d(1_{[N]},...,1_{[N]}) \geq \alpha^{\frac{d(d+1)}{2}} \times \frac{1}{N'^d} \times 2\left(\frac{N}{2}\right)^d,
\end{equation*}
because if we choose $1 \leq n_i \leq N/2$ for all $1 \leq i \leq d$, then for all $1 \leq i \leq j \leq d$, $1 \leq n_i + n_j \leq N$, giving at least $\left( \frac{N}{2} \right)^d$ $d$-configurations of the form $(n_i+n_j)_{1 \leq i \leq j \leq d}$, and if we choose $0 \leq n_i \leq N/2 -1$ for all $1 \leq i \leq d$, then for all $1 \leq i \leq j \leq d$, $1 \leq n_i + n_j +1 \leq N$, giving at least $\left( \frac{N}{2} \right)^d$ $d$-configurations of the form $(n_i+n_j+1)_{1 \leq i \leq j \leq d}$. Therefore we have at least $2\left(\frac{N}{2}\right)^d$ $d$-configurations in total (including trivial ones).

Recall that a $d$-configuration $(n_i + n_j)_{1 \leq i \leq j \leq d}$ is trivial if at least two of the $n_i$'s are equal. Let us find an upper bound for the number of trivial $d$-configurations of the form $(n_i+n_j)_{1 \leq i \leq j \leq d}$ in $A$. First, we have $d \choose 2$ ways to choose $i_0$ and $j_0$ for which we set $n_{i_0}=n_{j_0}$. We know that $2n_i$ must be in $A$ for all $i$, so there are at most $|A|=\alpha N$ choices for $n_{i_0}$ (we have an equality here if $A$ only contains even integers). Then $n_{j_0}$ is forced to be equal to $n_{i_0}$, and we have at most $|A|=\alpha N$ choices for each one the $d-2$ other $n_i$'s too, giving at most ${d \choose 2}(\alpha N)^{d-1}$ trivial $d$-configurations of the form $(n_i+n_j)_{1 \leq i \leq j \leq d}$ in $A$. Note that we may have counted some $d$-configurations that are not fully contained in $A$, but we only want an upper bound so this is not a problem here.
The $d$-configurations of the form $(n_i+n_j+1)_{1 \leq i \leq j \leq d}$ work exactly in the same way, and in total we have at most $2 {d \choose 2}(\alpha N)^{d-1} = d(d-1) (\alpha N)^{d-1}$ trivial $d$-configurations in $A$. Therefore, using the fact that $A$ contains fewer than $\frac{\alpha^{d(d+1)/2} N^d}{2^{d-1} + 1}$ non-trivial $d$-configurations, we obtain
\begin{equation*}
\Pi_d(1_A,...,1_A) \leq \frac{\alpha^{\frac{d(d+1)}{2}}}{2^{d-1} + 1} \left(\frac{N}{N'}\right)^d + d(d-1)\frac{(\alpha N)^{d-1}}{N'^d}.
\end{equation*}

After some calculation we find that if $N$ satisfies the condition $N > 16^d \alpha^{-\frac{d(d+1)}{2}}$ then the second term is negligible and we obtain $$\Pi_d(1_A,...,1_A) \leq \frac{2}{2^{d}+1} \alpha^{\frac{d(d+1)}{2}} \left(\frac{N}{N'}\right)^d.$$
Note that the bound we chose for $N$ is not optimal, but it has the advantage of not being too complicated and does not change the final bound in Theorem~\ref{th1}.

Therefore the sum of the $2^{d(d+1)/2}-1$ other terms involving $f_A$ must have magnitude at least $$\left(\frac{1}{2^{d-1}} - \frac{2}{2^{d}+1} \right) \alpha^{\frac{d(d+1)}{2}} \left(\frac{N}{N'}\right)^d.$$

Since $N' \leq 3N$, one of those terms must have magnitude larger than $$\frac{1}{3^d \times 2^{\frac{d(d+1)}{2} +2d}} \alpha^{\frac{d(d+1)}{2}}.$$ To avoid heavy expressions, we will say that $\mid \Pi_d((g_{ij})_{1 \leq i \leq j \leq d}) \mid \geq \left(\frac{\alpha}{C}\right)^{\frac{d(d+1)}{2}}$ for some absolute constant $C>0$.
\end{proof}

Now we will use Lemma~\ref{lemma1} and Theorem~\ref{config} to establish the following corollary.

\begin{corollary}
\label{cor1}
Let $\alpha$, $0< \alpha < 1$, be a real number. Suppose that $N > 16^d \alpha^{-\frac{d(d+1)}{2}}$ and that $A$ is a subset of $[N]$ with $|A|=\alpha N$ containing fewer than $\frac{\alpha^{d(d+1)/2} N^d}{2^{d-1} + 1}$ non-trivial $d$-configurations.
Let $f_A: \Z / N' \Z \rightarrow \R$ be the balanced function of $A$.
Then $$\parallel f_A \parallel_{U^2} \geq \left(\frac{\alpha}{C}\right)^{\frac{d(d+1)}{2}}$$ for some constant $C>0$.
\end{corollary}

\begin{proof}
Let $\Pi_d((g_{ij})_{1 \leq i \leq j \leq d})$ be the same term as in Lemma~\ref{lemma1}.

By Theorem~\ref{config}, we have
\begin{equation*}
\begin{aligned}
\mid \Pi_d((g_{ij})_{1 \leq i \leq j \leq d}) \mid &= \mid \mathbb{E}_{n_1,...,n_d \in \Z / N'\Z} \prod_{1 \leq i \leq j \leq d} g_{ij}(n_i + n_j) + \mathbb{E}_{n_1,...,n_d \in \Z / N'\Z} \prod_{1 \leq i \leq j \leq d} g_{ij}(n_i + n_j+1) \mid
\\&\leq \mid \mathbb{E}_{n_1,...,n_d \in \Z / N'\Z} \prod_{1 \leq i \leq j \leq d} g_{ij}(n_i + n_j) \mid + \mid \mathbb{E}_{n_1,...,n_d \in \Z / N'\Z} \prod_{1 \leq i \leq j \leq d} g_{ij}(n_i + n_j+1) \mid
\\&\leq 2\min_{1\leq i \leq j \leq d} \| g_{ij} \| _{U^{2}} \leq 2\parallel f_A \parallel_{U^{2}}.
\end{aligned}
\end{equation*}

And by Lemma~\ref{lemma1}, we have $$\left|\Pi_d((g_{ij})_{1 \leq i \leq j \leq d})\right| \geq \left(\frac{\alpha}{C}\right)^{\frac{d(d+1)}{2}}$$ for some constant $C>0$.

Therefore $$2\parallel f_A \parallel_{U^{2}} \geq |\Pi_d((g_{ij})_{1 \leq i \leq j \leq d})| \geq \left(\frac{\alpha}{C}\right)^{\frac{d(d+1)}{2}}.$$
\end{proof}

\subsection{Inverse results for the Gowers $U^2$-norm}
In this subsection, we use the fact that $f_A$ has a large Gowers $U^2$-norm to show that it also has a large Fourier coefficient.

Let $f$ be a function from $\Z/N\Z$ to $\C$. The \emph{Fourier transform} $\hat{f}$ of $f$ is defined, for all $\xi \in \Z/NZ$, by the formula:
\begin{equation*}
\hat{f}(\xi) := \mathbb{E}_{x \in \Z/N\Z} f(x) e(-\xi x/N).
\end{equation*}

Let us recall the following theorem proved in~\cite{Green}.

\begin{theorem}
\label{th2}
Suppose that $f: \Z / N' \Z \rightarrow \C$ is a $1$-bounded function with $\parallel f \parallel_{U^{2}} \geq \delta$.
Then there is some $r \in \Z/N'\Z$ such that $$| \hat{f}(r)| \geq \delta^2.$$
\end{theorem}

By Corollary~\ref{cor1} and Theorem~\ref{th2}, we have the following result.
\begin{proposition}
\label{large}
Let $\alpha$, $0< \alpha < 1$, be a real number. Suppose that $N > 16^d \alpha^{-\frac{d(d+1)}{2}}$ and that $A$ is a subset of $[N]$ with $|A|=\alpha N$ containing fewer than $\frac{\alpha^{d(d+1)/2} N^d}{2^{d-1} + 1}$ non-trivial $d$-configurations.
Let $f_A: \Z / N' \Z \rightarrow \R$ be the balanced function of $A$.
Then there is some $r \in \Z/N'\Z$ and some constant $C>0$ such that $$| \hat{f}_A(r) | \geq \left(\frac{\alpha}{C}\right)^{d(d+1)}.$$
\end{proposition}

\subsection{Obtaining a density increment}
In this subsection, we will use the fact that $f_A$ has a large Fourier coefficient to find a density increment on some arithmetic progression and thus complete the proof of Proposition~\ref{keyprop}.

Since $f$ is supported on $[N]$ and $N'=2N+1$, Proposition~\ref{large} immediately implies the next proposition.

\begin{proposition}
\label{prop2}
Let $\alpha$, $0< \alpha < 1$, be a real number. Suppose that $N > 16^d \alpha^{-\frac{d(d+1)}{2}}$ and that $A$ is a subset of $[N]$ with $|A|=\alpha N$ containing fewer than $\frac{\alpha^{d(d+1)/2} N^d}{2^{d-1} + 1}$ non-trivial $d$-configurations.
Let $f_A$ be the balanced function of $A$, considered now as a function on $[N]$.
Then there is some $\theta \in [0,1]$ and some constant $C>0$ such that
\begin{equation}
\label{eq1}
\left|\sum_{x \in [N]} f_A(x) e(\theta x) \right| \geq \left(\frac{\alpha}{C}\right)^{d(d+1)} N.
\end{equation}
\end{proposition}

Let us recall a lemma proved in~\cite{TaoVu}, adapted from Roth's original argument~\cite{Roth}.

\begin{lemma}
\label{roth_tao}
Let $f: \Z \rightarrow \R$ be a function supported on $[N]$ such that $|f(n)| \leq 1$ for all $n$, $\sum_n f(n) =0$ and $$\left|\mathbb{E}_{x \in [N]} f(x) e(\theta x)\right| \geq \sigma$$ for some $\theta \in [0,1]$ and $\sigma > 0$.
Then there exists a non-trivial arithmetic progression $P \subseteq [N]$ with $|P| \geq c \sigma^2 \sqrt{N}$ and $$\mathbb{E}_{x \in P} f(x) \geq \frac{\sigma}{4}.$$
\end{lemma}

We have $ -\alpha \leq f_A (n) \leq 1- \alpha$ for all $n$, so $|f_A(n)| \leq 1$ for all $n$. We also have $\sum_n f_A(n) =0$. Therefore, using the conclusion of Proposition~\ref{prop2}, we can apply Lemma~\ref{roth_tao} with $f = f_A$ and $\sigma = \left(\frac{\alpha}{C}\right)^{d(d+1)}$. We obtain an arithmetic progression $P \subseteq [N]$ of length $|P| \geq \left(\frac{\alpha}{C}\right)^{2d(d+1)} \sqrt{N}$ and $$\mathbb{E}_{x \in P} f_A(x) \geq \frac{1}{4} \left(\frac{\alpha}{C}\right)^{d(d+1)} \geq \left(\frac{\alpha}{C'}\right)^{d(d+1)},$$
which means that $$\frac{|A \cap P|}{|P|} \geq \alpha + \left(\frac{\alpha}{C'}\right)^{d(d+1)}.$$

We obtain the following proposition.
\begin{proposition}
\label{keyprop}
Suppose that $0 < \alpha < 1$ and $N > 16^d \alpha^{-\frac{d(d+1)}{2}}$. Suppose that $P \subseteq \Z$ is an arithmetic progression of length $N$ and that $A \subseteq P$ is a set with cardinality $\alpha N$. Then one of the following two alternatives holds:
\begin{itemize}
\item $A$ contains at least $\frac{\alpha^{d(d+1)/2} N^d}{2^{d-1} +1}$ $d$-configurations;
\item There is an arithmetic progression $P'$ of length $\geq \left(\frac{\alpha}{C}\right)^{2d(d+1)} N^{1/2}$ such that, writing $A':= A \cap P'$ and $\alpha' := |A'|/|P'|$, we have $\alpha' > \alpha + \left(\frac{\alpha}{C}\right)^{d(d+1)}$ for some absolute constant $C>0$.
\end{itemize}
\end{proposition}

\subsection{The final bound}
Finally, using Proposition~\ref{keyprop}, we obtain the following quantitative version of Theorem~\ref{th1}:
\begin{theorem}
\label{th1_gowers}
There is an absolute constant $C$ such that any subset $A \subseteq [N]$ with cardinality at least $\frac{CN}{(\log{\log{N}})^{ \frac{1}{d(d+1)-1}} }$ contains a non-trivial $d$-configuration.
\end{theorem}

\begin{proof}
Set $P_0 := [N]$ and assume that we have a set $A \subseteq P_0$ with $|A|= \alpha N$ containing no non-trivial $d$-configuration. Then we attempt to use Proposition~\ref{keyprop} repeatedly to obtain a sequence $P_0, P_1, ...$ of progressions and sets $A_i := A \cap P_i$. The densities $\alpha_i := |A_i|/|P_i|$ will satisfy $\alpha_{i+1} > \alpha_i + \left(\frac{\alpha_i}{C}\right)^{d(d+1)}$ and the length of $P_i$ will be $$N_i \geq \left(\frac{\alpha_{i-1}}{C}\right)^{2d(d+1)} N_{i-1}^{1/2} \geq \left(\frac{\alpha}{C}\right)^{2d(d+1) \sum_{k=0}^{i-1} \left(\frac{1}{2}\right)^k} N^{(1/2)^i} \geq \left(\frac{\alpha}{C}\right)^{4d(d+1)} N^{(1/2)^i},$$ using the fact that $\alpha_i \geq \alpha$ and $\sum_{k=0}^{i-1} \left(\frac{1}{2}\right)^k \leq 2$ for all $i \geq 1$.

But this iteration cannot last too long, otherwise we would obtain a set $A_i$ with density more than $1$ over the arithmetic progression $P_i$ for some $i$, which is impossible. In particular there cannot be more than $\left(\frac{C'}{\alpha}\right)^{d(d+1)-1}$ steps in total. We conclude that the applications of Proposition~\ref{keyprop} must have been invalid, which means that the condition $N_i > 16^d \alpha_i^{-\frac{d(d+1)}{2}}$ was violated. Since 
\begin{equation*}
N_i \geq \left(\frac{\alpha}{C}\right)^{4d(d+1)} N^{(1/2)^{\left(\frac{C'}{\alpha}\right)^{d(d+1)-1}}}
\end{equation*}
and $\alpha_i \geq \alpha$, we infer the bound
\begin{equation*}
16^d \alpha^{-\frac{d(d+1)}{2}} \geq \left(\frac{\alpha}{C}\right)^{4d(d+1)} N^{(1/2)^{\left(\frac{C'}{\alpha}\right)^{d(d+1)-1}}}.
\end{equation*}
Rearranging leads to
\begin{equation*}
\begin{aligned}
\log \log N &\leq \log\left(\log\left(16^d \alpha^{-\frac{d(d+1)}{2}}\right) + 4d(d+1) \log\left(\frac{C}{\alpha}\right)\right) + \log 2 \left(\frac{C'}{\alpha}\right)^{d(d+1)-1}
\\ &\leq \left(\frac{C''}{\alpha}\right)^{d(d+1)-1},
\end{aligned}
\end{equation*}
for some constant $C''$.

Therefore if $\alpha \geq \frac{C''}{(\log\log N)^{\frac{1}{d(d+1)-1}}}$ (ie. $N \geq \exp{\left(\exp{\left({\left(\frac{C''}{\alpha}\right)^{(d(d+1)-1)}}\right)}\right)}$), $A$ contains a non-trivial $d$-configuration. Theorem~\ref{th1_gowers} is proved.
\end{proof}

The bound obtained with this proof is better than the one using Szemer\'edi's theorem because we only have $3$ exponentials instead of $5$.

\section{The improvement in Sudakov, Szemer\'edi and Vu's theorem}
Now that we proved Theorem~\ref{th1_gowers}, we will use it to improve Sudakov, Szemer\'edi and Vu's theorem about sum-free subsets~\cite{Sudakov}.

Let $\phi(A)$ denote the maximum cardinality of a subset of $A$ which is sum-free with respect to $A$. Let $\phi(n)$ be the minimum of $\phi(A)$ over all sets $A$ of $n$ integers. Sudakov, Szemer\'edi and Vu obtained the first superlogarithmic lower bound for $\phi(n)$ by proving the following theorem~\cite{Sudakov}.

\begin{theorem}
\label{sudakovszemeredivu}
There is a function $g(n)$ tending to infinity with $n$ such that the following holds. Any set $A$ of $n$ integers contains a subset $B$ with cardinality $g(n) \log n$ such that $B$ is sum-free with respect to $A$.
\end{theorem}

Their proof shows that we can take $g(n)$ to be of the order $(\log^{(5)}n)^{1-o(1)}$. Thus they proved that $$\phi(n) \geq \log n \left(\log^{(5)}n\right)^{1-o(1)}.$$

It is easier to describe $g(n)$ as the inverse of an iterative exponential function. They set $g(n)$ in Theorem~\ref{sudakovszemeredivu} to be $c(m/\log m)$, where $c$ is a sufficiently small positive constant, $m = F^{-1}(n^{1/2})$ and $$F(h) = \exp \left( h^{182} \times \left( 2 \uparrow \left(e^{h^{32770}} \right) \uparrow 2 \uparrow 2 \uparrow (2h +9) \right) \right).$$

We will show that, using Theorem~\ref{th1_gowers}, we can set $g(n)$ to be $c(m/\log m)$, where $m = G^{-1}(n^{1/2})$ and $$G(h) = \exp \left( h^{182} \times \left( e \uparrow \left(c e^{h^{32770}} \right) \uparrow (h(h+1)-1) \right) \right).$$

In their proof, Sudakov, Szemer\'edi and Vu deduce Theorem~\ref{sudakovszemeredivu} from the following theorem.

\begin{theorem}
\label{sudakov2}
Let $X$, $Y$ be two finite sets of positive integers with $\frac{1}{h^{29}} |Y| \geq |X| \geq F(h)$, where $F$ is the function described above and $h$ is a sufficiently large integer. Then $Y$ contains a subset $Z$ of size $h$ which is disjoint from $X$ and is sum-free with respect to $X \cup Y$.
\end{theorem}

We will show that a modification of the proof of Theorem~\ref{sudakov2} allows us to replace $F$ by $G$ and obtain the following theorem.

\begin{theorem}
\label{sudakov2_mieux}
Let $X$, $Y$ be two finite sets of positive integers with $\frac{1}{h^{29}} |Y| \geq |X| \geq G(h)$, where $G$ is the function described above and $h$ is a sufficiently large integer. Then $Y$ contains a subset $Z$ of size $h$ which is disjoint from $X$ and is sum-free with respect to $X \cup Y$.
\end{theorem}

In particular, the proof of Theorem~\ref{sudakov2} uses the following corollary, proved using Szemer\'edi's theorem.

\begin{corollary}
\label{Vu}
If $A \subseteq \{1,...,N\}$ is a set of size $\alpha N$, and if $N > 2\uparrow 2 \uparrow (\alpha)^{-1} \uparrow 2 \uparrow 2 \uparrow (2k+9)$, then there is a subset $A' \subseteq A$ of $k$ elements such that, for any two elements $x,y \in A'$, there is an element $z$ of $A$ satisfying $x+y=2z$.
\end{corollary}

But by using Theorem~\ref{th1_gowers}, we can replace Corollary~\ref{Vu} by the following corollary, which comes with a better bound.

\begin{corollary}
\label{cor_meilleur}
If $A \subseteq \{1,...,N\}$ is a set of size $\alpha N$, and if $N > e \uparrow e \uparrow \left( \frac{C}{\alpha}\right) \uparrow (k(k+1)-1)$, then there is a subset $A' \subseteq A$ of $k$ elements such that, for any two elements $x,y \in A'$, there is an element $z$ of $A$ satisfying $x+y=2z$.
\end{corollary}
\begin{proof}
By Theorem~\ref{th1_gowers}, $A$ contains a non-trivial $k$-configuration $(n_i + n_j + a)_{1 \leq i \leq j \leq k}$. Now we can take $A'=\{2n_1 +a, ..., 2n_k+a\}$, and so for any two elements $2n_i+a$ and $2n_j +a \in A'$, $(2n_i+a) + (2n_j +a) = 2(n_i +n_j +a)$ with $n_i +n_j +a \in A$.
\end{proof}

To prove Theorem~\ref{sudakov2_mieux} we leave all Sudakov, Szemer\'edi and Vu's proof of Theorem~\ref{sudakov2} unchanged, except at the end, where we replace their bound by ours. We advise the reader to read the proof in~\cite{Sudakov}, because we won't copy the beginning of the proof here.
Thus when they need $m_1 \geq 2 \uparrow 2 \uparrow \left(e^{h^{32770}} \right) \uparrow 2 \uparrow 2 \uparrow (2h +9)$, we only need $m_1 \geq e \uparrow e \uparrow \left(C e^{h^{32770}}\right) \uparrow (h(h+1)-1)$.
As in their proof, we have $\log m_1 \geq \frac{\log |Y|}{h^{182}}.$ Then we only need to verify that $\log |Y| \geq h^{182} \times \left( e \uparrow \left(C e^{h^{32770}} \right) \uparrow (h(h+1)-1) \right).$ Since the right-hand side is equal to $\log G(h)$, this inequality follows from the assumption of Theorem~\ref{sudakov2_mieux} that $|Y| \geq |X| \geq G(h)$. This completes the proof.

Now we can derive Theorem~\ref{sudakovszemeredivu} from Theorem~\ref{sudakov2_mieux} in the same way as in~\cite{Sudakov}, except that we replace $F$ by $G$.

After some calculation, we obtain that we can take $g(n)$ to be of the order $(\log^{(3)}n)^{\frac{1}{32772} -o(1)}$.
Therefore we obtain the following stronger version of Theorem~\ref{sudakovszemeredivu}.

\begin{theorem}
There is a function $g(n)$ of the order $(\log^{(3)}n)^{\frac{1}{32772} -o(1)}$ such that any set $A$ of $n$ integers contains a subset $B$ with cardinality $g(n) \log n$ such that $B$ is sum-free with respect to $A$.
\end{theorem}
%
%

\section{Conclusion}
We generalised Roth's theorem to $d$-configurations and showed that any set $A \subseteq \{1,...,N\}$ with density $\alpha$ such that $N > e \uparrow e \uparrow \left( \frac{C}{\alpha}\right) \uparrow (d(d+1)-1)$ contains a non-trivial $d$-configuration. Then we used this result to improve Sudakov, Szemeredi and Vu's theorem about sum-free subsets and proved that $\phi(n) \geq \log n \left(\log ^{(3)} n \right)^{1/32772 - o(1)}$, which is the best lower bound known to date for $\phi(n)$.

Bourgain~\cite{Bourgain} modified Roth's original Fourier analytic proof~\cite{Roth} of Roth's theorem by increasing the density of $A$ on Bohr sets instead of arithmetic progressions. By doing so, he improved Roth's bound $N \geq \exp\left(\exp\left(\frac{C}{\alpha}\right)\right)$ and showed that $N \geq \left(\frac{C}{\alpha}\right)^{C'/\alpha^2}$ suffices. Therefore it should be possible to do a similar modification to our proof in order to obtain a stronger version of Theorem~\ref{th1} with a bound of the type $N \geq \left(\frac{c(d)}{\alpha}\right)^{\left(\frac{c'(d)}{\alpha}\right)^{d(d-1)}}.$ This would hopefully improve our result about sum-free subsets and lead to a bound of the form $\phi(n) \geq \log n (\log^{(2)}n)^{c -o(1)}.$
Even if the technical details of such a proof might be considerable, it would constitute an interesting subject for further research. However this lower bound is still far from the best upper bound currently known, $\phi(n) \leq O (e ^{ \sqrt{\log n}})$, proved by Ruzsa in~\cite{Ruzsa}, so we can assume that many interesting results about sum-free subsets are still to be found.

\section*{Acknowledgements}
The author would like to thank Ben Green for introducing her to this very interesting subject and for the precious advices he gave her during the elaboration of this work.
She also thanks Fernando Shao for drawing her attention to a calculation mistake in an early version of this article.

\bibliographystyle{siam}
\bibliography{biblio2}

\end{document}